\documentclass[preprint,review,12pt]{elsarticle}

\usepackage{amssymb}
\usepackage{amsthm}
\usepackage{amsfonts}
\usepackage{amsmath}

\newtheorem{theorem}{Theorem}[section]

\newtheorem{corollary}[theorem]{Corollary}

\journal{\tt  arXiv.org}

\begin{document}

\begin{frontmatter}



\title{Randomly Weighted Averages: A Multivariate Case}

\author{Hazhir Homei}

\address{Department of Statistics, Faculty of Mathematical
Sciences, \\ University of Tabriz,  P.O.Box 51666--17766, Tabriz, IRAN. \\
                       {\tt homei@tabrizu.ac.ir}}

\begin{abstract}
Stochastic linear combinations of some random vectors are studied where the distribution of the random vectors and the joint distribution of their coefficients are Dirichlet.
A method is provided for calculating the distribution of these combinations which has been studied before by some authors.    Our main result is a generalization of
some existing results with a simpler proof.
 \end{abstract} 
 
\bigskip

\end{frontmatter}



\section{Introduction}
For given random variables $X_1,\cdots,X_n$ the distribution of the stochastic linear combination $Z\!=\!\sum_{i=1}^nW_iX_i$  is used for the problems in lifetime, stochastic matrices, neural networks and other applications
in sociology and biology.
Let $X_i$ ($1\leqslant\!i\leqslant\!n$) be the lifetime measured in a lab and $0\!\leqslant\!W_i\!\leqslant\!1$ be the random effect of the environment on it; so  $W_iX_i\!\leqslant\!X_i$ and thus
$\sum_{i=1}^nW_iX_i$ is the average
lifetime in the environment (see Homei~(2015)).
Recently, several authors have focused on computing the lifetime of systems in the real conditions.
Indeed, the randomly linear combination of random vectors  have   many applications including
 traditional portfolio selection models,  relationship between attitudes and behavior, number of cancer cells in tumor biology, stream flow in hydrology (Nadarajah \& Kotz~(2005)),  branching processes, infinite particle systems and probabilistic algorithms, and vehicle speed and  lifetime (cf.~Homei~(2015), Rezapour \&  Alamatsaz~(2014) and the references therein)  so finding their distributions has attracted the attentions of numerous researchers.

In this paper, considering the dependent components and the real environmental conditions, the distribution of lifetimes (in case of the Dirichlet distributions) is calculated
showing that
the  main result obtained here covers several previous results and provides a simpler proof (cf. e.g.
Johnson \& Kotz~(1990a), Sethuraman~(1994), van Assche~(1987), Volodin \&  Kotz \&  Johnson~(1993)).

The inner product of two random vectors was introduced in Homei~(2014) and the exact distribution of this product was investigated for some random vectors with Beta and Dirichlet distributions. In this paper   a new generalization for the inner product of two random vectors is introduced.
For a random vector ${\bf W}'=\langle W_1,\cdots,W_n\rangle$ and a vector $\underline{{\bf X}}=\langle {\bf X}_1,\cdots,{\bf X}_n\rangle$ of random vectors (each ${\bf X}_i$ being  $k$-dimensional) the inner product of ${\bf X}$ and ${\bf W}$ is essentially the linear transformation of ${\bf W}$ under the $k\times n$ matrix $\underline{{\bf X}}$ which is  ${\bf Z}\!=\!\sum_{i=1}^n W_i{\bf X}_i$.
 We  assume that
  ${\bf W}$ is  independent from ${\bf  X}_1,\cdots,{\bf X}_n$ and   has  Dirichlet
     distribution; also  ${\bf X}_i$'s  have  Dirichlet  distributions.
Identifying the distribution of ${\bf Z}$  usually requires
\begin{itemize}
\item[(i)] either some long computations with combinatorial identities  (see e.g.
    Volodin \&  Kotz \&  Johnson~(1993), Johnson \& Kotz~(1990), Johnson \& Kotz~(1990a), Sethuraman~(1994), Homei~(2014));
\item[(ii)] or advanced techniques requiring performing certain transformations and solving differential equations (see e.g. Homei~(2012), Soltani \& Homei (2009), van Assche~(1987)  or Homei~(2015) and the  references therein).
\end{itemize}

\noindent
In this paper a new way is introduced to identify the distribution of randomly linear combinations (of Dirichlet  distributions) by which a class of stochastic differential equations can be solved; this new method is much simpler than the existing ones.

\section{The Main Result}
Our main result identifies the distribution of the randomly linear combination when the coefficients come from a  Dirichlet  distribution.

\begin{theorem}\label{main}
If  ${\bf X}_1,\cdots,{\bf X}_n$ are  independent $k$-variate   random vectors with respectively  $Dirichlet(\alpha^{(1)}), \cdots,Dirichlet(\alpha^{(n)})$ distributions, for  some $k$-dimensional vectors $\alpha^{(j)}=\langle \alpha^{(j)}_1,\cdots,\alpha^{(j)}_k\rangle$ \textup{(}$j=1,\cdots,n$\textup{)}, and the  random vector ${\bf W}=\langle W_1,\cdots,W_n\rangle$  is  independent from ${\bf X}_1,\cdots,{\bf X}_n$ and has the distribution   $$Dirichlet\left(\sum_{i=1}^{k}\alpha_i^{(1)},\cdots,
\sum_{i=1}^{k}\alpha_i^{(n)}\right),$$ then
the randomly linear combination  ${\bf Z}=\sum_{i=1}^n W_i{\bf X}_i$ has the distribution $$Dirichlet\left(\sum_{j=1}^{n}\alpha_1^{(j)},\cdots,
\sum_{j=1}^{n}\alpha_k^{(j)}\right).$$
\end{theorem}

\begin{proof}[\textup{(}The First\textup{)} Proof]
Let $Y_j$ ($j=1,\cdots,n$) be  independent random variables independent from $({\bf X}_1,\cdots,{\bf X}_n)$  that have the   distribution $\Gamma(\sum_{i=1}^k\alpha^{(j)}_i,\frac{1}{\mu})$, respectively. Thus,   $$W_j\overset{d}{=}\frac{Y_j}{\big({\sum_{l=1}^nY_l}\big)} \textrm{ for } j=1,\cdots,n.$$
  The $k$-dimensional vectors ${\bf T}_j=Y_j{\bf X}_j$  have  independent components and $$\langle\Gamma(\alpha^{(j)}_1,\frac{1}{\mu}),\cdots,
\Gamma(\alpha^{(j)}_k,\frac{1}{\mu})\rangle$$ distribution. Thus, we have
 $$\sum_{j=1}^n\big(\frac{{\bf T}_j}{\sum_{l=1}^nY_l}\big)
 =\sum_{j=1}^n\frac{Y_j}{\sum_{l=1}^nY_l}{\bf X}_j\overset{d}{=}\sum_{j=1}^n W_j{\bf X}_j.$$
Now,    $\sum_{j=1}^n\big(\frac{{\bf T}_j}{\sum_{l=1}^nY_l}\big)$ has Dirichlet$(\sum_{j=1}^{n}\alpha_1^{(j)},
\cdots,\sum_{j=1}^{n}\alpha_k^{(j)})$ distribution,  so   $\sum_{j=1}^n W_j{\bf X}_j={\bf Z}$ has the same distribution.
\end{proof}

\begin{corollary}\label{cor}
Let the $k$-variate random variables ${\bf X}_1,\cdots,{\bf X}_n$ be independent and with common distribution. If ${\bf W}=\langle W_1,\cdots,W_n\rangle$ is independent from ${\bf X}_1,\cdots,{\bf X}_n$ having $Dirichlet(k\alpha,\cdots,k\alpha)$ distribution, then the randomly linear combination ${\bf Z}=\sum_{i=1}^n W_i{\bf X}_i$ has the $Dirichlet(n\alpha,\cdots,n\alpha)$ distribution if and only if ${\bf X}_i$'s ($i=1,\cdots,n$) have  $Dirichlet(\alpha,\cdots,\alpha)$ distributions.
\end{corollary}

\begin{proof}
 The ``if" part follows from Theorem~\ref{main} and the ``only if" part follows from the theorem of Volodin \&  Kotz \& Johnson~(1993).
\end{proof}

Some similar results (to Corollary~\ref{cor}) can be found in e.g. Alamatsaz~(1993).
The following results are obtained as  special cases of
Theorem~\ref{main} and Corollary~\ref{cor}:

$\bullet\,$ Lemma~1 of Sethuraman~(1994) (for $n\!=\!2$), and

$\bullet\,$   Theorem of Volodin \&  Kotz \& Johnson~(1993)   (when $n=k$ and $\alpha^{(j)}_i\!=\!a$ for $j\!=\!1,\!\cdots\!,k$, $i\!=\!1,\!\cdots\!,k$),

\noindent which discusses the multivariate case, or the references below which discuss the single-variable case (note that for $k\!=\!2$ the Dirichlet distribution leads to the Beta distribution):


$\bullet\,$ Theorem of van Assche~(1987) (for $n\!=\!k\!=\!2, \alpha^{(1)}_1\!
=\!\alpha^{(1)}_2\!=\!\alpha^{(2)}_1\!=\!\alpha^{(2)}_2\!
=\!\frac{1}{2}$),

$\bullet\,$   Theorem~2 of Johnson \& Kotz~(1990a) (for $n\!=\!k\!=\!2,\alpha^{(1)}_1\!
=\!\alpha^{(1)}_2\!=\!\alpha^{(2)}_1\!=\!\alpha^{(2)}_2\!=\!a$),

$\bullet\,$ Theorem~2.4 of  Homei~(2014) (for $k\!=\!2$ and $\alpha_1^{(j)}=\alpha,\alpha_2^{(j)}=1-\alpha$ for $j=1,\cdots,n$),

$\bullet\,$ Theorem~1 of  Homei~(2013) (for $k\!=\!2$);

\noindent
and others; see the references in Homei~(2015).

\section{Four Other Proofs for Theorem~\ref{main} and a Variant of It}
\subsection{Moment Generation Method}
\begin{proof}[The Second Proof]
The generating moment function  of ${\bf T}_j$'s in the first proof are
$${\rm E}(e^{{\bf t}'{\bf T}_j})={\rm EE}
(e^{{\bf t}'Y_j{\bf X}_j}\mid X_j)={\rm E}(\frac{1}{1-{\bf t}'{\bf X}_j})^{\sum_{i=1}^{k}\alpha_{i}^{(j)}}.$$
By \cite[page 77]{kmr} we have $${\rm M}_{{\bf T}_j}{({\bf t})}=(\frac{1}{1-t_j})^{\sum_{i=1}^{k}\alpha_{i}^{(j)}}.$$
So, the components of the vector ${\bf T}_j$ are independent and have gamma distributions, which proves the theorem.
\end{proof}
\subsection{Applying Basu's Theorem}
\begin{proof}[The Third Proof]
We can write (for $j=1,\cdots,n$)
$${\bf X}_j \sim \left(\frac{\Gamma_{1j}(\alpha_1^{(j)})}
{\sum_{i=1}^{k}\Gamma_{ij}(\alpha_i^{(j)})},\cdots,
\frac{\Gamma_{kj}(\alpha_k^{(j)})}
{\sum_{i=1}^{k}\Gamma_{ij}(\alpha_i^{(j)})}\right),$$
and
$${\bf W}\sim \left(\frac{\sum_{i=1}^k\Gamma_{ij}(\alpha_i^{(1)})}
{\sum_{j=1}^{n}\sum_{i=1}^{k}\Gamma_{ij}
(\alpha_i^{(j)})},\cdots,
\frac{\sum_{i=1}^k\Gamma_{ij}(\alpha_i^{(n)})}
{\sum_{j=1}^{n}\sum_{i=1}^{k}\Gamma_{ij}
(\alpha_i^{(j)})}\right).$$
So,
$${\bf Z}\overset{d}{=}\left(\frac{\sum_{j=1}^{k}
\Gamma_{1j}(\alpha_1^{(j)})}
{\sum_{j=1}^{n}\sum_{i=1}^{k}\Gamma_{ij}
(\alpha_i^{(j)})},\cdots,
\frac{\sum_{j=1}^{k}\Gamma_{kj}(\alpha_k^{(j)})}
{\sum_{j=1}^{n}\sum_{i=1}^{k}\Gamma_{ij}
(\alpha_i^{(j)})}\right).$$
Let us recall that $\Gamma_{ij}(\alpha_i^{(j)})$'s (for $i=1,\cdots,k$ and $j=1,\cdots,n$) are independent random variables with gamma distributions, which implies the independence of the components of ${\bf W}$ from ${\bf X}_j$'s (Basu's Theorem).
\end{proof}
\subsection{Mathematical Induction}
\begin{proof}[The Fourth Proof]
For $n=2$ the theorem follows from \cite[Lemma~1]{sethuraman}. Suppose the theorem holds for $n$ (the induction hypothesis). We prove it for $n+1$ (the induction conclusion): By dividing the both sides of
$$\sum_{i=1}^{n+1}Y_i{\bf X}_i=\big(\sum_{i=1}^nY_i\big)\left(\sum_{j=1}^n\frac{Y_j}
{\sum_{l=1}^nY_l}{\bf X}_j\right)+Y_{n+1}{\bf X}_{n+1}$$
by $\sum_{i=1}^{n+1}Y_i$ and using the induction hypothesis we have
$$\sum_{i=1}^{n+1}\frac{Y_i}{\sum_{i=1}^{n+1}Y_i}{\bf X}_i=\big(\sum_{i=1}^n\frac{Y_i}{\sum_{i=1}^{n+1}Y_i}
\big)\left(\sum_{j=1}^n\frac{Y_j}
{\sum_{l=1}^nY_l}{\bf X}_j\right)+\frac{Y_{n+1}}{\sum_{i=1}^{n+1}Y_i}{\bf X}_{n+1}$$
in which the right hand side holds by \cite[Lemma~1]{sethuraman} (for $n=2$).
\end{proof}
\subsection{The Moments Method}
\begin{proof}[The Fifth Proof]
The general moments $(s_1,s_2,\cdots ,s_k)$ of ${\bf Z}$ are as follows:
$${\rm E}\!\!\left(\prod_{j=1}^k\Big(\sum_{i=1}^n W_jX_{ij}\Big)^{s_j}\right)\!\!=\!{\rm E}\!\!\left(\prod_{j=1}^k\Big(\sum_{h_j}{s_j\choose h_{1j}, h_{2j}, \cdots ,
h_{nj}} \prod_{i=1}^n (W_jX_{ij})^{h_{ij}}\Big)\right)$$
where $\sum_{h_j}$ denotes summation over all non-negative integers
$$h_j=(h_{1j},h_{2j},\cdots ,h_{nj}) \textrm{ subject to }
\sum_{i=1}^{n}h_{ij}=s_j \quad  (j=1,2,\cdots ,n).$$
This
equation  can be rearranged as
$$={\rm E}\left(\sum_{h_1}\cdots \sum_{h_k}\Big(\prod_{j=1}^k{s_j\choose h_{1j}, h_{2j},\cdots,
h_{nj}} \prod_{j=1}^k \prod_{i=1}^n (W_iX_{ij})^{h_{ij}}\Big)\right)$$
$$={\rm E}\left(\sum_{h_1}\cdots \sum_{h_k}\Big(\prod_{j=1}^k{s_j\choose h_{1j}, h_{2j},\cdots,
h_{nj}} (\prod_{i=1}^n W_i^{h_{i\ast}})\prod_{j=1}^k\prod_{i=1}^n X_{ij}^{h_{ij}}\Big)\right)$$
(where $h_{i\ast}=\sum_{j=1}^kh_{ij}$) and also
$$=\sum_{h_1}\cdots \sum_{h_k}\left(\prod_{j=1}^k{s_j\choose h_{1j}, h_{2j},\cdots ,
h_{nj}}{\rm E}\Big(\prod_{i=1}^n W_i^{h_{i\ast}}\Big){\rm E}\Big(\prod_{j=1}^k\prod_{i=1}^n X_{ij}^{h_{ij}}\Big)\right)\quad (\ddag)$$
By  the Dirichlet distribution we have
$${\rm E}\left(\prod_{i=1}^nW_i^{h_{i\ast}}\right)=\frac{\Gamma (\sum_{i=1}^n\sum_{j=1}^k\alpha_j^{(i)})}{\Gamma (\sum_{i=1}^n\sum_{j=1}^k\alpha_j^{(i)}+\sum_{j=1}^k s_j)}\prod_{i=1}^n\frac{\Gamma (\sum_{j=1}^k\alpha_j^{(i)}+h_{i\ast})}{\Gamma (\sum_{j=1}^k\alpha_j^{(i)}}$$
and also
$${\rm E}\left(\prod_{j=1}^k\prod_{i=1}^n X_{ij}^{h_{ij}}\right)=\prod_{i=1}^n{\rm E}\left(\prod_{j=1}^k X_{ij}^{h_{ij}}\right),$$
and again by  the Dirichlet distribution
$${\rm E}\left(\prod_{j=1}^k X_{ij}^{h_{ij}}\right)=\frac{\Gamma (\sum_{j=1}^k\alpha_j^{(i)})}{\Gamma (\sum_{j=1}^k\alpha_j^{(i)}+h_{i\ast})}
\prod_{j=1}^k\frac{\Gamma (\alpha_j^{(i)}+h_{ij})}{\Gamma (\alpha_j^{(i)})}.$$
So, by using $(\ddag)$
$$=\sum_{h_1}\cdots \sum_{h_k}\prod_{j=1}^k {s_j\choose h_{1j}, h_{2j},\cdots ,h_{nj}}\Big(\frac{\Gamma (\sum_{i=1}^n\sum_{j=1}^k\alpha_j^{(i)})}{\Gamma (\sum_{i=1}^n\sum_{j=1}^k\alpha_j^{(i)}+\sum_{j=1}^k s_j)}$$
$$\prod_{i=1}^n\frac{\Gamma (\sum_{j=1}^k\alpha_j^{(i)}+h_{i.})}{\Gamma (\sum_{j=1}^k\alpha_j^{(i)}}\Big)
\left(\prod_{i=1}^n\frac{\Gamma (\sum_{j=1}^k\alpha_j^{(i)})}{\Gamma (\sum_{j=1}^k\alpha_j^{(i)}+h_{i.})}
\prod_{j=1}^k\frac{\Gamma (\alpha_j^{(i)}+h_{ij})}{\Gamma (\alpha_j^{(i)})}\right)$$
$$=\frac{\Gamma (\sum_{i=1}^n\sum_{j=1}^k\alpha_j^{(i)})}{\Gamma (\sum_{i=1}^n\sum_{j=1}^k\alpha_j^{(i)}+\sum_{j=1}^k s_j)}\sum_{h_1}\cdots \sum_{h_k}\prod_{j=1}^k {s_j\choose h_{1j}, h_{2j},\cdots ,h_{nj}}$$
$$\prod_{j=1}^k\prod_{i=1}^n
\frac{\Gamma(\alpha_j^{(i)}+h_{ij})}
{\Gamma(\alpha_j^{(i)})}.$$
By considering the fact that the sum of the Dirichlet-multimonial distributions on their support equals to one, we have
$$=\frac{\Gamma (\sum_{i=1}^n\sum_{j=1}^k\alpha_j^{(i)})}{\Gamma (\sum_{i=1}^n\sum_{j=1}^k\alpha_j^{(i)}+\sum_{j=1}^k s_j)}\prod_{j=1}^k
\frac{\Gamma(\sum_{i=1}^n\alpha_j^{(i)}+s_j)}
{\Gamma(\sum_{i=1}^n\alpha_j^{(i)})}$$
which is the general moment of the k-variate
$$Dirichlet\left(\sum_{i=1}^n \alpha_1^{(i)},\sum_{i=1}^n \alpha_2^{(i)},\cdots ,\sum_{i=1}^n\alpha_k^{(i)}\right)$$
distribution, and since ${\bf Z}$ is a bounded random variable, its distribution is uniquely determined by its moments. Thus the proof is complete.
\end{proof}

\subsection{A Variant of Theorem~\ref{main}}

\begin{theorem}
The distribution of the randomly linear combination  ${\bf Z}=\sum_{i=1}^n W_i{\bf X}_i$ is $$Dirichlet\left(\frac{1}{2}+\sum_{i=1}^{n}\alpha_i\right),$$ where ${\bf X}_1,\cdots,{\bf X}_n$ are two-dimensional  independent multivariate  random vectors with $$Dirichlet\Big(\frac{1}{2}+\alpha_1\Big), \cdots,Dirichlet\Big(\frac{1}{2}+\alpha_n\Big)$$ distributions and the  random vector ${\bf W}=\langle W_1,\cdots,W_n\rangle$  is  independent from $({\bf X}_1,\cdots,{\bf X}_n)$ and has the distribution   $$Dirichlet(\alpha_1,\cdots,
\alpha_n).$$
\end{theorem}
\begin{proof}
Let $Y_j$ ($j=1,\cdots,n$) be  independent random variables independent from $({\bf X}_1,\cdots,{\bf X}_n)$  that have the   distribution $\Gamma(\alpha_j,\mu)$, respectively.
It can be seen, by some classic ways (e.g. ${\rm E}\big(e^{{\bf t}'{\bf T}}\big)=\big[\Psi({\bf t})\big]^{\sum_j\alpha_j}$ from Kubo \& Kuo \&  Namli~(2013), Table~2), that the distribution of ${\bf T}(=\sum_j{\bf T}_j=\sum_j Y_j{\bf X}_j)$ is the same distribution of  ${\bf T}_j$
with the parameter $\sum_j\alpha_j$.
  We can also write ${\bf T}$ as
$${\bf T}=\sum_jY_j{\bf X}_j=\left(\sum_iY_i\right)
\left(\sum_j\frac{Y_j}{\sum_iY_i}X_j\right)$$
and so we have ${\bf T}=Y{\bf Z}$ in which $Y$
has the gamma distribution with the parameter $\sum_j\alpha_j$,
and ${\bf T}$ has the $F$ distribution with the parameter $\sum_j\alpha_j$, and $Y$ and ${\bf Z}$ are independent from each other.
Of course, one can define ${\bf T}'=Y'{\bf X}'$  in such a way that ${\bf T}'\overset{d}{=}{\bf T}$, $\sum_jY_j\overset{d}{=}Y'$ and ${\bf X}'\sim Dirichlet\left(\frac{1}{2}+\sum_j\alpha_j\right)$.
One can conclude that ${\bf Z}$ and ${\bf X}'$ have identical distributions by calculating the general moments $(s_1,s_2)$ of ${\bf T}$ and ${\bf T}'$, i.e., ${\rm E}\big({Z}_1^{s_1}{Z}_2^{s_2}\big)={\rm E}\big(({X}_1')^{s_1}({X}_2')^{s_2}\big)$.
\end{proof}

Actually, the above proof also shows that:

\begin{theorem}
Let $Y_j$ \textup{(}$j=1,\cdots,n$\textup{)} be  independent random variables independent from $({\bf X}_1,\cdots,{\bf X}_n)$  that have the   distribution $\Gamma(\alpha_j,\mu)$, respectively, where ${\bf X}_i$'s are independent from each other and have Dirichlet distributions. If  ${\bf X}$ has a bounded support and the independent random variable $Y$ has $\Gamma(\sum_j\alpha_j,\mu)$ distribution such that $\sum_iY_i{\bf X}_i\overset{d}{=}Y{\bf X}$, then ${\bf X}$ and ${\bf Z}=\sum_i W_i{\bf X}_i$ have identical distributions, where ${\bf W}=\langle W_1,\cdots,W_n\rangle$ is independent from ${\bf X}_i$'s and has Dirichlet$(\alpha_1,\cdots,\alpha_n)$ distribution.
\end{theorem}

\section{Some Applications in Stochastic Differential Equations}
In this section, using Theorem~\ref{main} and Corollary~\ref{cor}, we prove some interesting mathematical facts.
As an example consider the following differential equation for each $n$ (cf. Homei~(2014)):
$$(1)\quad \frac{(-1)^{n-1}}{(n-1)!}
\frac{n-1}{2}\frac{d^{n-1}}{dz^{n-1}}\int_0^1 (1-t)^{(n-3)/2}(z^2-t)^{-1/2}dt
=\left(\frac{1}{\sqrt{z^2-1}}\right)^n,$$
which could be of interest for some authors, who first guess the solution and then, by using techniques like Leibniz differentiations or change of variables or integration by part, prove that it   satisfies    the equation (by some long inductive arguments).


Let us recall that
 Theorem~1 of Homei~(2015) identifies the distribution of (the $1$-dimensional) ${Z}$ from the distributions of ${X}_i$'s by means of the differential equation:
$$(2)\qquad \frac{(-1)^{n^*-1}}{(n^*-1)!}
\frac{d^{n^*-1}}{dz^{n^*-1}}{\cal S}(F_{Z},z)=\prod_{i=1}^n
\frac{(-1)^{m_i-1}}{(m_i-1)!} \frac{d^{m_i-1}}{dz^{m_i-1}}
{\cal S}(F_{X_i},z),$$
where $F_Y$ denotes the cumulative distribution function of a random variable $Y$ and ${\cal S}(F_Y,z)$ is  defined by ${\cal S}(F_Y,z)=\int_{\mathbb{R}} \frac{1}{z-x}F_Y(dx)$ for $z \in \mathbb{C}\cap ({\rm supp} F_Y)^\complement$ in which ${\rm supp} F_Y$  stands for the support of $F_Y$.


Let the distribution of ${\bf X}_i$ ($i=1,\cdots,n$) be Arcsin, that is ${\cal S}(F_{X_i},z)=\frac{1}{\sqrt{z^2-1}}$. Also, let $m_i=1$ ($i=1,\cdots,n$) and $n^\ast=n$. Then from the equation~(2) we will have
$$(3)\qquad \qquad \frac{(-1)^{n-1}}{(n-1)!}\frac{d^{n-1}}{dz^{n-1}}{\cal S}(F_{Z},z)=\left(\frac{1}{\sqrt{z^2-1}}\right)^n.$$
The solution of the equation~(3) identifies  the Stieltjes transformation of the distribution of ${Z}$. Alternatively, from Theorem~\ref{main} (or Corollary~\ref{cor}) the distribution of $Z$ is power semicircle (see Homei~(2014) for more details).
Since the Stieltjes transformation of the power semicircle distribution is $\frac{n-2}{2}\int_0^1 (1-t)^{(n-3)/2}(z^2-t)^{-1/2}dt$ (see e.g. Arizmendi \& Perez-Abreu~(2010))
then by substituting it in the equation~(2) we will get the equation~(1) immediately.

As another example  consider the moment generating function on  the vector ${\bf T}_j$:
$${\rm M}_{{\bf T}_j}({\bf t}')={\rm E}\big(\exp({{\bf t}'Y_j{\bf X}_j})\big).$$
Using the double conditional expectation and the fact that the components of ${\bf T}_j$ are independent with gamma distributions  we have
$${\rm E}\big(\frac{1}{1-{\bf t}'{\bf X}_j}\big)^{\sum_{i=1}^{k}\alpha_i^{(j)}}=
\prod_{i=1}^k\big(\frac{1}{1-{t}_i}\big)^{\alpha_i^{(j)}}$$
which proves   Proposition~4.4 of  Kerov \&  Tsilevich~(2004) (cf. also Karlin \& Micchelli \&   Rinott~(1986), page~77).

\section{Conclusions}
\noindent Some sporadic works of other authors have been unified here; the main result (Theorem~\ref{main}) has several other different proofs (available upon request) each of which can have various applications (the five proofs presented here are dedicated to all the family members of Professor {\sc A.R. Soltani}).
Our method reveals the advantage of the method of Stieltjes transforms for identifying the distribution of stochastic linear combinations, first used by van Assche~(1987) and later generalized by others.

\end{document}